\documentclass[a4paper,11pt]{amsproc}

\usepackage{enumerate}% for roman enumerate style

\usepackage{xargs, xcolor}
\usepackage{amsmath,amssymb,amsfonts,amsthm}
\usepackage{mathtools, todonotes}

\usepackage[backrefs]{amsrefs}

\usepackage[onehalfspacing]{setspace}

\newcommandx{\set}[2][2=\empty]{\{#1\ifx#2\empty\else\,|\,#2\fi\}}% set builder notation
\newcommand{\norm}[1]{\lvert#1\rvert}% norm of an element
\newcommand{\card}[1]{\left\lvert#1\right\rvert}% cardinality
\newcommandx{\gensubgrp}[2][2=\empty]{\langle#1\ifx#2\empty\else\,|\,#2\fi\rangle}% generated subgroup
\newcommandx{\gennorsubgrp}[2][2=\empty]{\langle\!\langle#1\ifx#2\empty\else\,|\,#2\fi\rangle\!\rangle}
\newcommandx{\gensubsp}[2][2=\empty]{\langle#1\ifx#2\empty\else\,|\,#2\fi\rangle}% generated subspace
\newcommand{\nats}{\mathbb{N}}% natural numbers
\newcommand{\finfield}{\mathbb{F}}% finite field
\DeclareMathOperator{\id}{id}% identity
\DeclareMathOperator{\supp}{supp}% support
\newcommand{\floor}[1]{\left\lfloor#1\right\rfloor}% floor operator
\newcommand{\ceil}[1]{\left\lceil#1\right\rceil}% ceiling operator
\DeclareMathOperator{\fix}{fix}% fixed points of a map
\newcommand{\C}{\mathbf{C}}% centralizer
\newcommand{\calU}{\mathcal{U}}% calligraphic U
\newcommand{\freegrp}{\mathbf F}% free group
\DeclareMathOperator{\diam}{diam}% diameter
\DeclareMathOperator{\ev}{ev}% evaluation map
\DeclareMathOperator{\PSL}{PSL}% projective special linear group
\DeclareMathOperator{\cn}{cn}% covering number of a finite simple group
\DeclareMathOperator{\cd}{cd}% covering diameter of a finite simple group

\theoremstyle{plain}
\newtheorem{theorem}{Theorem}
\newtheorem{lemma}{Lemma}
\newtheorem{corollary}{Corollary}

\theoremstyle{definition}
\newtheorem{remark}{Remark}

\title{Word maps with constants on symmetric groups}
\author{Jakob Schneider}
\address{Jakob Schneider, TU Dresden, 01062 Dresden, Germany}
\email{jakob.schneider@tu-dresden.de}
\author{Andreas Thom}
\address{Andreas Thom, TU Dresden, 01062 Dresden, Germany}
\email{andreas.thom@tu-dresden.de}

\begin{document}

\begin{abstract}
	We study word maps with constants on symmetric groups. Even though there are non-trivial mixed identities of bounded length that are valid for all symmetric groups, we show that no such identities can hold in the limit in a metric sense. Moreover, we prove that word maps with constants and non-trivial content, that are short enough, have an image of positive diameter, measured in the normalized Hamming metric, which is bounded from below in terms of the word length. Finally, we also show that every self-map $G\to G$ on a finite non-abelian simple group is actually a word map with constants from $G$.
\end{abstract}

\dedicatory{To Doris, Hannah, Torsten, Jonathan, and Paul who supported me \\ during the long time of my sickness. J.\ Schneider}
\maketitle
\onehalfspacing

\section*{Introduction}

Recently, there has been increasing interest in word maps and laws on finite, algebraic, and topological groups \cites{avnigelanderkassabovshalev2013word,bandmangariongrunewald2012surjectivity,elkasapythom2014goto,garionshalev2009commutator,gordeev1997freedom,gordeevkunyavskiiplotkin2016word,gordeevkunyavskiiplotkin2018word,guralnickliebeckobrienshalevtiep2018surjective,guralnicktiep2015effective,huilarsenshalev2015waring,klyachkothom2017new,larsen2004word,larsenshalev2009word,larsenshalev2016distribution,larsenshalev2017words,larsenshalevtiep2012waring,lubotzky2014images,schneider2019phd,tomanov1985generalized}. Here, every word $w\in\freegrp_r=\gensubgrp{x_1,\ldots,x_r}$ induces a word map $w\colon G^r\to G$ on every group $G$ by substitution, where $\freegrp_r$ denotes the free group in the $r$ generators $x_1,\ldots, x_r$ and we set $w(g_1,\ldots,g_r)$ to be the image of $w$ under the unique homomorphism $\freegrp_r\to G$ which maps $x_i\mapsto g_i$ ($i=1,\ldots,r$). The word $w$ is called an \emph{identity} or a \emph{law} for $G$ iff $w(g_1,\ldots,g_r)=1_G$ for all choices of the $g_i\in G$ ($i=1,\ldots,r$). It is an interesting question to study the length of the shortest non-trivial law of a given finite group $G$. This was done in \cites{thom2017length,bradford2019short}. A bit less restrictively, one can ask when the image of a word map is small in a metric sense, see for example \cites{thom2013convergent, lubotzky2014images, larsen2004word}.

In this article we study word maps with constants. A word with constants in $G$ is an element of the free product $\freegrp_r\ast G$. We get an associated map $G^n\to G$ in a similar way as before by replacing the variables with elements from $G$. The word is called a \emph{mixed identity} or a \emph{law with constants} iff $w(g_1,\ldots,g_r)=1_G$ for all choices of the $g_i\in G$ ($i=1,\ldots,r$). 

Following \cites{klyachkothom2017new,nitschethom2022universal}, for $w\in\freegrp_r\ast G$, we study the \emph{augmentation} $\varepsilon(w)\in\freegrp_r$, which replaces all constants by the neutral element. We call $\varepsilon(w)\in\freegrp_r$ the \emph{content} of $w$. It has been observed in different circumstances \cites{klyachkothom2017new,nitschethom2022universal} that word maps with constants which have a non-trivial content, tend to have a large image. We prove a corresponding result for symmetric groups with respect to the normalized Hamming metric. For $\sigma\in S_n$, we write 
$$
    \fix(\sigma)\coloneqq\set{x}[x.\sigma=x]
$$ 
for the set of \emph{fixed points} of $\sigma$, 
$$
    \supp(\sigma)\coloneqq\set{x}[x.\sigma\neq x]
$$ 
for its \emph{support}, and 
$$
    \norm{\sigma}\coloneqq\card{\supp(\sigma)}
$$ 
for its (unnormalized) \emph{Hamming norm}. The \emph{normalized Hamming norm} on $S_n$ is then defined as $\norm{\sigma}_{\rm H}\coloneqq\frac{1}{n}\norm{\sigma}$. Also, for $\sigma,\tau\in S_n$ let $d(\sigma,\tau)\coloneqq\norm{\sigma^{-1}\tau}$ be their \emph{Hamming distance} and $d_{\rm H}(\sigma,\tau)\coloneqq\frac{1}{n}\norm{\sigma^{-1}\tau}$ be their \emph{normalized Hamming distance}. For $S\subseteq S_n$ write 
$$
    \diam(S)\coloneqq\max_{\sigma,\tau\in S}{d(\sigma,\tau)}
$$ 
for the \emph{diameter} of the set $S$.

We say that $w\in\freegrp_r\ast G$ is \emph{strong} if, when written in normal form, removing the constants does not lead to any cancellation among the variables -- in particular, if $w$ is strong, then either the content is non-trivial or the word was a constant from $G$.

%If $G$ has a center, then this length is two, since then $[x,c]$ is a mixed identity, where $c\in\Z(G)\setminus\trivgrp$ and $\freegrp_1=\gensubgrp{x}$. In particular, all nilpotent groups admit a mixed identity of length two. Hence we look at groups without center. 
There are some obvious mixed identities for $S_n$. E.g.\ take $\tau$ a transposition, then $[x,\tau]^6$ is such an identity, since $[g,\tau]=\tau^{-g}\tau$ is the product of two transpositions and so either a $3$-cycle, the product of two disjoint transpositions, or trivial, where $g\in S_n$ is arbitrary. We notice here that the word with constants $[x,\tau]^6$ has a trivial content. It is one consequence of our main result, Theorem~\ref{thm:main}, that this must be the case for any mixed identity on $S_n$ of \emph{small} length. However, there are long words with non-trivial content that are an identity for $S_n$, e.g.\ take $x^{n!}$. We show that a non-trivial word map with constants coming from a short word can have a small image (i.e.\ with small diameter) only if there are constants involved with small support. 

Fix a group $G$ of possible constants. Every word $w\in\freegrp_r\ast G$ can be written uniquely in the form
$$
    w=c_0 x_{i(1)}^{\varepsilon(1)} c_1 \cdots c_{l-1} x_{i(l)}^{\varepsilon(l)} c_l\in\freegrp_r\ast G,
$$ 
where $x_1,\ldots,x_r$ are the free generators of $\freegrp_r$, $\varepsilon(j)=\pm1$, $i(j)\in\set{1,\ldots,r}$ ($j=1,\ldots,l$), and $c_j\in G$ ($j=0,\ldots,l$) are such that there is no cancellation, i.e.\ $x_{i(j)}^{\varepsilon(j)}=x_{i(j+1)}^{-\varepsilon(j+1)}$ implies that $c_j\neq 1_G$ for $j=1,\ldots,l-1$. Then, we set $\ell(w)=l$ and call it the \emph{length} of $w$. Note that $\ell$ is just the word length where we give elements from $G$ length zero.

\vspace{0.2cm}
Our main result is the following theorem.

\begin{theorem}\label{thm:main}
    Let $w \in \freegrp_r \ast S_n$ and consider the associated word map with constants $w \colon S_n^r \to S_n$. Then the following holds:
    \begin{enumerate}[(i)]
        \item If $w$ has non-trivial content, then 
%$$
%\frac{\diam(w(S_n^r))+1}{n}\geq\frac{1}{2\ell(w)}(1+4\ell(w))^{-\floor{\ell(w)/2}}%.
%$$
        $$
            \frac{\diam(w(S_n^r))+1}{n}\geq \frac{1}{2}\exp(-\log(5\ell(w)) \ell(w)/2).
        $$
        \item If $w$ is strong and $w \notin S_n$, then
        $$
            \frac{\diam(w(S_n^r))+1}{n} \geq\frac{1}{2\ell(w)}.
        $$
    \end{enumerate}
    More precisely, if $w\notin S_n$ is arbitrary and the inequality in $(ii)$ is violated, then a constant of size at most $2(\diam(w(S_n^r))+1)\ell(w)$ can be found, so that the removal of that constant leads to cancellation of variables.
\end{theorem}

\begin{corollary}\label{cor:dense_diam}	
    Assume that $w\in\freegrp_r\ast S_n$ ($n\geq 2$) is strong and induces a constant map $S_n^r\to S_n$. Then either $w\in S_n$ or $\ell(w) \geq \frac{n}{2}$.
\end{corollary}

The shortest known strong mixed identities are the laws constructed in \cite{kozmathom2016divisibility}, they are of length $\exp(C \log(n)^4 \log\log(n)).$ It is an interesting open problem to close this gap.

The second part of Theorem \ref{thm:main} implies the following corollary which is of independent interest.

\begin{corollary} \label{cor:ultra}
    A metric ultraproduct of symmetric groups $S_n$ equipped with the normalized Hamming metric, with $n$ tending to infinity along the chosen ultrafilter, does not satisfy any non-trivial mixed identity.
\end{corollary}

For completeness, let us briefly recall the notion of the \emph{metric ultraproduct} of a family of metric groups $(G_n,d_n)_{n\in\nats}$ with respect to an \emph{ultrafilter} $\calU$ on $\nats$. Here $d_n$ is a bi-invariant metric on $G_n$ ($n\in\nats$). 
It is defined as the quotient $\prod_{n\in\nats}{G_n}\big/N_\calU$, where $N_\calU\coloneqq\set{(g_n)_{n\in\nats}}[d_n(g_n,1_{G_n})\to_\calU 0]$ is the normal subgroup of null sequences, and is itself a metric group when equipped with the limit metric.

For more on metric ultraproducts of symmetric groups, see \cites{elekszabo2003sofic,pestov2008hyperlinear,nikolovschneiderthom2018some}. Note that the previous corollary is in contrast to the usual ultraproduct of symmetric groups, which satisfies a mixed identity $[x,\tau]^6$, where $\tau$ is an ultraproduct of transpositions as discussed above.

\vspace{0.2cm}
Despite the restrictions on images of word maps with constants given by the length of the words and the structure of the constants involved, we prove in the Appendix that every self-map of a non-abelian simple group is a word map with constants.

\section{Basic definitions}

In this short section, we make some basic definitions, which are needed later. Let $G$ be a group, the group of possible constants. As above, write the fixed word $w\in\freegrp_r\ast G$ uniquely in the form
$$
    w=c_0 x_{i(1)}^{\varepsilon(1)} c_1 \cdots c_{l-1} x_{i(l)}^{\varepsilon(l)} c_l\in\freegrp_r\ast G,
$$ 
where $\varepsilon(j)=\pm1$, $i(j)\in\set{1,\ldots,r}$ ($j=1,\ldots,l$), and $c_j\in G$ ($j=0,\ldots,l$) are such that there is no cancellation. 

We define the sets of indices $J_0(w),J_+(w),J_-(w)\subseteq\set{1,\ldots,l-1}$ by $$J_0(w)\coloneqq\set{j}[i(j)\neq i(j+1)],$$ $$J_+(w)\coloneqq\set{j}[i(j)=i(j+1) \mbox{ and }\varepsilon(j)=\varepsilon(j+1)],$$ and $$J_-(w)\coloneqq\set{j}[i(j)=i(j+1)\mbox{ and }\varepsilon(j)=-\varepsilon(j+1)].$$ Note that $J_0(w),J_+(w),J_-(w)$ partition the set $\set{1,\ldots,l-1}$.
We call $J_-(w)$ the set of \emph{critical indices} of $w$ since the removal of a constant $c_j$ for $j\in J_-(w)$ leads to cancellation among the variables. The constant $c_j$ is then called a \emph{critical constant}. We call $v$ an \emph{elementary reduction} of $w\in\freegrp_r\ast G$, if it is obtained from $w$ by deleting a critical constant and reducing the outcome.

Now, we focus on word maps with constants in symmetric groups. Keep the notation from above and set $G\coloneqq S_n$. For simplicity, we assume that $c_0=1_G=\id$. Keep the definition of the sets $J_-(w)$, $J_+(w)$, and $J_0(w)$. Define 
$$
    w_j\coloneqq x_{i(1)}^{\varepsilon(1)}c_1\cdots x_{i(j)}^{\varepsilon(j)}c_j
$$ 
to be the $j$th prefix of $w$ ($j=0,\ldots,l$). Note that $w_l=w$ by construction. For a reduced word $v\in \freegrp_r\ast S_n$, write $\norm{v}_i$ for its \emph{$i$-length}, i.e.\ the number of occurrences of $x_i$ and $x_i^{-1}$ in its reduced representation ($i=1,\ldots,r$). Let $w(S_n^r)\subseteq S_n$ denote the \emph{image} of the \emph{word map} $w\colon S_n^r\to S_n$ induced by $w$. 

\section{Proof of the main result}

We are now ready to state the main technical lemma:

\begin{lemma}\label{lem:no_lw_lrg_supp_consts}
	In this setting, assume that $l>0$ and let $d\geq 1$ be an integer such that the following hold:
	\begin{enumerate}[\normalfont(i)]		
		\item $n\geq (d-1)(\norm{w}_{i(j)}+\norm{w}_{i(j+1)})+\norm{w_j}_{i(j)}+\norm{w_{j+1}}_{i(j+1)}-1$  for all $j\in J_0(w)$;	
		\item $n\geq 2((d-1)\norm{w}_{i(j)}+\norm{w_{j+1}}_{i(j)}-1)$ for $j\in J_+(w)$;
		\item $\norm{c_j}\geq2((d-1)\norm{w}_{i(j)}+\norm{w_{j+1}}_{i(j)})-3$ for all $j\in J_-(w)$;
		\item $n\geq d\norm{w}_{i(l)}+1$.
	\end{enumerate}
	Then, we have that $\diam(w(S_n^r))\geq d$.
\end{lemma}

To prove this lemma and the results thereafter, we need some additional graph-theoretical terminology: By a \emph{directed graph} $G$, we mean an object that consists of a set of \emph{vertices} $V(G)$ and a set of \emph{edges} $E(G)$ together with maps $\alpha_{+1},\alpha_{-1}\colon E(G)\to V(G)$. For an edge $e$ of $G$, we call $\alpha_{+1}(e)$ its \emph{source} and $\alpha_{-1}(e)$ its \emph{target}; also for $\varepsilon\in\set{\pm1}$, we call $\alpha_\varepsilon(e)$ the \emph{$\varepsilon$-source} and $\alpha_{-\varepsilon}(e)$ the \emph{$\varepsilon$-target} of $e$. Let $S$ be a set of \emph{labels}. By an edge-labeling of $G$ by $S$ we mean a mapping $\lambda\colon E(G)\to S$. If $G$ is equipped with such a labeling, we call it a \emph{directed $S$-edge-labeled graph}. In this case, an edge $e$ of $G$ which is labeled by $s\in S$ is called an \emph{$s$-arrow}. In this situation, if we exhibit a particular vertex $v$ of $G$, we call $(G,v)$ a \emph{pointed directed $S$-edge-labeled graph}.

We will also make use of the following notation for partially defined maps: If $a\colon A\to B$ and $b\colon B\to C$ are such maps, then their \emph{composition} $ab\colon A\to C$ is the partially defined map given by $x.ab=(x.a).b$ if $x.a$ exists and $(x.a).b$ exists, and is undefined otherwise.

\vspace{0.2cm}
Here comes the proof of the previous lemma:

\begin{proof}[Proof of Lemma~\ref{lem:no_lw_lrg_supp_consts}]
	We have to prove that, under the assumptions of the lemma, there exist two elements $\sigma,\tau\in w(S_n^r)$ such that $d(\sigma,\tau)\geq d$. So let $\tau\in w(S_n^r)$ be an arbitrary element. In the following, we construct $\sigma\in w(S_n^r)$ such that $d(\sigma,\tau)\geq d$:
	
	Let $S_n$ act on the $n$-element set $\Omega$ and fix $d$ arbitrary distinct points $\omega_1,\ldots,\omega_d\in\Omega$ (this is certainly possible by condition (iv) in the lemma, since $n\geq d\norm{w}_{i(l)}+1>d$ as $l\geq 1$). Now we inductively construct a family $(G_k^j,\omega_k^j)_{j,k}$ of pointed directed $S$-edge-labeled graphs on the fixed vertex set $V(G_k^j)\coloneqq\Omega$ where $S=\set{x_1,\ldots,x_r}$ is the set of generators of $\freegrp_r$ ($j=0,\ldots,l$, $k=1,\ldots,d$). The graphs $G_k^j$ will be constructed in the following lexicographic order:
	$$G_1^0,G_1^1,G_1^2,\ldots,G_1^l=G_2^0,G_2^1,G_2^2\ldots,G_2^l=G_3^0,G_3^1,G_3^2,\ldots\ldots,G_d^l.$$ 
	Each such graph $G_k^j$ will be a \emph{partial Schreier graph} of $\freegrp_r$ (i.e.\ a graph that can be completed to a Schreier graph of $\freegrp_r$) which is obtained from its predecessor by adding a single $x_{i(j)}$-arrow $e_k^j$, starting from the empty graph $G_1^0$, i.e.\ the graph with $E(G_1^0)=\emptyset$. Therefore, there will be precisely $((k-1)\norm{w}_{i}+\norm{w_j}_{i})$-many $x_{i}$-arrows in $G_k^j$ which encode partial injective maps $\pi_{k,i}^j\colon\Omega\to\Omega$ ($i=1,\ldots,r$). With this notation, the points $\omega_k^j$ will be chosen in such a way that 
	$$
	    \omega_k^j=\omega_k.w_j(\pi_{k,1}^j,\ldots,\pi_{k,r}^j)\quad\text{and}\quad\omega_k.w_l(\pi_{k,1}^l,\ldots,\pi_{k,r}^l)\neq\omega_k.\tau
	$$ 
	for $j=0,\ldots,l$ and $k=1,\ldots,d$, i.e.\ if $j\geq 1$ we have 
	$$
	    \alpha_{\varepsilon(j)}(e_k^j)=\omega_k^{j-1}\quad\text{and}\quad \alpha_{-\varepsilon(j)}(e_k^j).c_j=\omega_k^j.
	$$ 
	Hence, setting $\pi_1,\ldots,\pi_r\in S_n$ to be extensions of $\pi_{d,1}^l,\ldots,\pi_{d,r}^l$, the points $\omega_k^j$ ($j=0,\ldots,l$) will be precisely the \emph{trajectory} of $\omega_k$ under the prefixes of $w(\pi_1,\ldots,\pi_r)$. Thus, setting $\sigma\coloneqq w(\pi_1,\ldots,\pi_r)$, we will have that $\omega_k.\sigma\neq\omega_k.\tau$ for all $k=1,\ldots,d$, so that $d(\sigma,\tau)\geq d$ as desired.
	
	We are left to carry out the construction of the family $(G_k^j,\omega_k^j)_{j,k}$ of pointed $S$-edge-labeled graphs. Recall that we start with $G_1^0$ being the empty graph, hence $E(G_1^0)=\emptyset$ and $\omega_1^0\coloneqq\omega_1,\ldots,\omega_d^0\coloneqq\omega_d$. Assume that we are to construct $(G_k^j,\omega_k^j)$ out of the previous data. If $j=0$, there is nothing to do. So assume $j\geq 1$ and we are to add the $x_{i(j)}$-arrow $e^j_k.$ We assume by induction that we are already given an admissible `starting point' $\alpha_{\varepsilon(j)}(e_k^j)=\omega_k^{j-1}$ of our edge $e_k^j$, i.e.\ $\omega_k^{j-1}\neq\alpha_{\varepsilon(j)}(e)$ for any $x_{i(j)}$-arrow $e\in E(G_k^{j-1})$. Our task is to find an admissible `end point' $\alpha_{-\varepsilon(j)}(e_k^j)$ of $e_k^j$. This means, we have to ensure that the following conditions are satisfied:	
	\begin{enumerate}[(a)]
		\item $\alpha_{-\varepsilon(j)}(e_k^j)\neq\alpha_{-\varepsilon(j)}(e)$ for any $x_{i(j)}$-arrow $e\in E(G_k^{j-1})$;
		\item if $i(j)=i(1)$ and $\varepsilon(j)=-\varepsilon(1)$, then $\alpha_{-\varepsilon(j)}(e_k^j)\neq\omega_m$ for all $m=k+1,\ldots,d$;
		\item if $j<l$, then $\omega_k^j\coloneqq\alpha_{-\varepsilon(j)}(e_k^j).c_j\neq\alpha_{\varepsilon(j+1)}(e)$ for any $x_{i(j+1)}$-arrow $$e\in E(G_k^j)\coloneqq E(G_k^{j-1})\cup\set{e_k^j};$$
		\item if $j=l$, then $\omega_k^j\coloneqq\alpha_{-\varepsilon(j)}(e_k^j).c_j\neq\omega_k.\tau$.
	\end{enumerate}
	Let us briefly explain this: (a) means that $e_k^j$ will not have the same $\varepsilon(j)$-target as any $x_{i(j)}$-arrow in $G_k^{j-1}$, ensuring that $G_k^j$ remains a partial Schreier graph. (b) is necessary, since the $\varepsilon(1)$-sources of the $x_{i(1)}$-arrows $e_m^1$ are already fixed to be $\omega_m$ from the beginning (for all $m=1,\ldots,d$), so these cannot be used by another $x_{i(1)}$-arrow as its $\varepsilon(1)$-source. (c) is necessary to ensure that $e_k^{j+1}$ will have a valid `starting point' $\omega_k^j=\alpha_{\varepsilon(j+1)}(e_k^{j+1})$ (when $j<l$) which is not already in use as the $\varepsilon(j+1)$-source of another $x_{i(j+1)}$-arrow in $G_k^j$. Finally, (d) ensures that $\omega_k.\sigma\neq\omega_k.\tau$.
	
	Now we count the number of possibilities to choose the $\varepsilon(j)$-target $\alpha_{-\varepsilon(j)}(e_k^j)$ of $e_k^j$ according to (a)--(d): As $G_k^{j-1}$ is a partial Schreier graph, no two of its $x_{i(j)}$-arrows have the same $\varepsilon(j)$-target. Hence there are precisely 
	$$
	n-((k-1)\norm{w}_{i(j)}+\norm{w_{j-1}}_{i(j)})=n-((k-1)\norm{w}_{i(j)}+\norm{w_j}_{i(j)}-1)
	$$ 
	possible $\varepsilon(j)$-targets for $e_k^j$ that satisfy (a). In the following three cases, we assume that $j<l$, so that (d) is irrelevant.
	\vspace{1ex}
	
	\emph{Case~(i): $i(j)\neq i(j+1)$.} This means that $j\in J_0(w)$. Then $G_k^j$ has precisely 
	$$
	(k-1)\norm{w}_{i(j+1)}+\norm{w_j}_{i(j+1)}=(k-1)\norm{w}_{i(j+1)}+\norm{w_{j+1}}_{i(j+1)}-1
	$$
	$x_{i(j+1)}$-arrows, none of which is $e_k^j$, so (c) rules out as many vertices for the $\varepsilon(j)$-target of $e_k^j$. In the worst case, the condition in (b) is satisfied, so that (b) rules out at most $d-k$ more vertices. In total, we get at least
	\begin{eqnarray*}
		&&n-((k-1)(\norm{w}_{i(j)}+\norm{w}_{i(j+1)})+\norm{w_j}_{i(j)}+\norm{w_{j+1}}_{i(j+1)}-2+d-k)\\
		&=&
		n-((k-1)(\norm{w}_{i(j)}+\norm{w}_{i(j+1)}-1)+\norm{w_j}_{i(j)}+\norm{w_{j+1}}_{i(j+1)}+d-3)\\
		&\geq&
		n-((d-1)(\norm{w}_{i(j)}+\norm{w}_{i(j+1)}-1)+\norm{w_j}_{i(j)}+\norm{w_{j+1}}_{i(j+1)}+d-3)\\
		&=&
		n-((d-1)(\norm{w}_{i(j)}+\norm{w}_{i(j+1)})+\norm{w_j}_{i(j)}+\norm{w_{j+1}}_{i(j+1)}-2)
	\end{eqnarray*}
	possible choices for the $\varepsilon(j)$-target of $e_k^j$, which is a positive number by assumption (i) in the lemma.
	\vspace{1ex}
	
	\emph{Case (ii): $i(j)=i(j+1)$ and $\varepsilon(j)=\varepsilon(j+1)$.} This means that $j\in J_+(w)$. Considering (c), there are precisely 
	$$
	(k-1)\norm{w}_{i(j)}+\norm{w_{j-1}}_{i(j)}=(k-1)\norm{w}_{i(j)}+\norm{w_{j+1}}_{i(j)}-2
	$$
	$x_{i(j+1)}=x_{i(j)}$-arrows $e$ different from $e_k^j$ in $G_k^j$ (namely the ones in $G_k^{j-1}$) which admit as many vertices as their $\varepsilon(j+1)=\varepsilon(j)$-sources, and the condition $\alpha_{-\varepsilon(j)}(e_k^j).c_j\neq\alpha_{\varepsilon(j)}(e_k^j)$ rules out the vertex $\alpha_{\varepsilon(j)}(e_k^j).c_j^{-1}=\omega_k^{j-1}.c_j^{-1}$. In the worst case, the condition in (b) is satisfied, so that (b) rules out at most $d-k$ further vertices. In total, we get at least
	\begin{eqnarray*}
		&&n-(2((k-1)\norm{w}_{i(j)}+\norm{w_{j+1}}_{i(j)}-2)+1+d-k)\\
		&=&
		n-((k-1)(2\norm{w}_{i(j)}-1)+2\norm{w_{j+1}}_{i(j)}-4+d)\\
		&\geq&
		n-((d-1)(2\norm{w}_{i(j)}-1)+2\norm{w_{j+1}}_{i(j)}-4+d)\\
		&=&
		n-(2((d-1)\norm{w}_{i(j)}+\norm{w_{j+1}}_{i(j)})-3)
	\end{eqnarray*}
	possible choices for $\alpha_{-\varepsilon(j)}(e_k^j)$, which is positive by assumption~(ii) in the lemma.
	\vspace{1ex}
	
	\emph{Case~(iii): $i(j)=i(j+1)$ and $\varepsilon(j)=-\varepsilon(j+1)$.} This means that $j\in J_-(w)$. As in the previous case, (c) splits into two parts: The $x_{i(j+1)}=x_{i(j)}$-arrows different from $e_k^j$ in $G_k^j$ rule out 
	$$
	(k-1)\norm{w}_{i(j)}+\norm{w_{j-1}}_{i(j)}=(k-1)\norm{w}_{i(j)}+\norm{w_{j+1}}_{i(j)}-2
	$$
	vertices for the $\varepsilon(j)$-target of $e_k^j$, and the condition $\alpha_{-\varepsilon(j)}(e_k^j).c_j\neq\alpha_{-\varepsilon(j)}(e_k^j)$ rules out the $n-\norm{c_j}$ fixed points of $c_j$. In the worst case, the condition in (b) is satisfied, and (b) rules out $d-k$ further vertices. Hence we have at least
	\begin{eqnarray*}
		&& n-(2((k-1)\norm{w}_{i(j)}+\norm{w_{j+1}}_{i(j)}-2)+n-\norm{c_j}+d-k)\\
		&=&
		\norm{c_j}-((k-1)(2\norm{w}_{i(j)}-1)+2\norm{w_{j+1}}_{i(j)}-5+d)\\
		&\geq&
		\norm{c_j}-((d-1)(2\norm{w}_{i(j)}-1)+2\norm{w_{j+1}}_{i(j)}-5+d)\\
		&=&
		\norm{c_j}-(2((d-1)\norm{w}_{i(j)}+\norm{w_{j+1}}_{i(j)})-4)
	\end{eqnarray*}
	choices for $\alpha_{-\varepsilon_j}(e_k^j)$, which is positive by assumption~(iii) in the lemma.
	\vspace{1ex}
	
	\emph{Case~(iv): $j=l$.} Then (c) is irrelevant, (b) can rule out at most $d-k$ vertices for $\alpha_{-\varepsilon_l}(e_k^l)$ if ($\ast$) holds, and (d) rules out one further vertex. Hence in total we have at least
	$$
	n-((k-1)\norm{w}_{i(l)}+\norm{w_l}_{i(l)}-1+d-k+1)\geq n-d\norm{w}_{i(l)}
	$$
	vertices remaining. This is a positive number by assumption~(iv) of the lemma. Hence the proof is complete.
\end{proof}

\begin{remark}\label{rmk:est_shrp}
	The condition on $\diam(w(S_n^r))$ is almost sharp: Take $g$ to be the $c$-cycle $(1\cdots c)$ and set $w=g^{x_1}$. Then $\diam(w(S_n^r))=\min\set{2c,n}$ and the bound in the lemma gives $\diam(w(S_n^r))\geq\floor{c/4}$.
\end{remark}

We set $$\norm{w}_\infty\coloneqq\max_{i\in\set{1,\ldots r}}{\norm{w}_i} \quad \mbox{and} \quad \norm{w}_{\rm crit}\coloneqq \min \left\{n, \min_{j \in J_-(w)}  \norm{c_j}\right\}.$$
The four assumptions in Lemma~\ref{lem:no_lw_lrg_supp_consts} can obviously by strengthened to the single inequality $\norm{w}_{\rm crit}\geq 2d \norm{w}_\infty$. It follows that 
$$
\diam(w(S_n^r))\geq\floor{\frac{\norm{w}_{\rm crit}}{2\norm{w}_\infty}}>\frac{\norm{w}_{\rm crit}}{2\norm{w}_\infty}-1
$$ 
for $n\geq2$.
%\begin{enumerate}[(i')]
%	\item $n\geq d(\norm{w}_i+\norm{w}_j)$ for all $i\neq j$;
%	\item $n\geq 2d\norm{w}_i$ for all $i$;
%	\item $\norm{c}\geq 2d\norm{w}_i$ for all $c\in\crit_i(w)$ and all $i$;
%	\item[(iv)] $n\geq d\norm{w}_{i(l)}+1$.
%\end{enumerate}
Note that $\norm{w}_{\infty}$ is obviously bounded by the length $\ell(w)$. We will use the above inequality in the form 
$$
\norm{w}_{\rm crit} \leq 2\cdot(\diam(w(S_n^r))+1)\cdot\ell(w).
$$

We interpret this inequality as follows: If the diameter of the word image is small, then either there is a small critical constant or the word length is large.

\begin{remark}
    We suppose that a similar lemma can be proved for families like $\PSL_n(p)$ for a fixed prime $p$ and $n$ growing, leading to results comparable to our main theorem. However, we note that for example for the family $\PSL_2(q)$ for $q$ growing, the length of the shortest mixed identity grows linearly in $q$. Note that in this case, the fact that the length of the shortest mixed identity must tend to infinity is already a consequence of the results in \cite{gordeevkunyavskiiplotkin2016word}. This will be the subject of further work.
\end{remark}

\begin{proof}[Proof of Theorem~\ref{thm:main}]

    The bound in (ii) directly follows from the inequality $\norm{w}_{\rm crit}\leq 2(\diam(w(S_n^r))+1)\ell(w)$, since when $w$ is strong we have $\norm{w}_{\rm crit}=n$ as there is no critical constant. The assertion about existence of small critical constants is also an immediate consequence of the inequality above.

    In order to prove (i) we study the effect of removal of critical constants in more detail.
    Let $v$ be an {elementary reduction} of $w\in\freegrp_r\ast G$, i.e.\ it is obtained from $w$ by deleting a critical constant and reducing the outcome.
    Note that removing the smallest critical constant $c$ does not change the diameter of the word image by more than $2\norm{c}=2\norm{w}_{\rm crit}$, so that given a small diameter, we can try to iterate elementary reductions. Let $w$ have non-trivial content and let $w=w_0,\ldots,w_m$ be a chain of words, where $w_i$ is an elementary reduction of $w_{i-1}$ ($i\geq 1$) by a smallest critical constant, so that there is no such reduction for $w_m$, i.e.\ $w_m$ is strong. Note that $w_i$ no longer denotes the $i$th prefix of $w$. Then $\ell(w_i)\leq\ell(w_{i-1})-2\leq\ell(w)-2i$. We have
    $$
        \norm{w_{i-1}}_{\rm crit}\leq 2\ell(w_{i-1})(\diam(w_{i-1}(S_n^r))+1)
    $$
    and so
    \begin{eqnarray*}
        \diam(w_i(S_n^r)) &\leq&\diam(w_{i-1}(S_n^r))+2\norm{w_{i-1}}_{\rm crit}\\
        &\leq& \diam(w_{i-1}(S_n^r))+4\ell(w_{i-1})(\diam(w_{i-1}(S_n^r))+1).
    \end{eqnarray*}
    It follows that
    $$
        \diam(w_i(S_n^r))+1\leq(1+4\ell(w_{i-1}))(\diam(w_{i-1}(S_n^r))+1).
    $$
    Hence, we obtain
    $$
        \diam(w_i(S_n^r))+1\leq (1+4\ell(w))^i(\diam(w(S_n^r))+1)
    $$
    by induction on $i=0,\ldots,m$, since $\ell(w_i)\leq\ell(w)$. By assumption, $w_m$ is not a constant from $G=S_n$. Then by the above inequality
    \begin{align*}
        n=\norm{w_m}_{\rm crit}&\leq 2(\diam(w_m(S_n^r))+1)\ell(w_m)\\
        &\leq2(1+4\ell(w))^m(\diam(w(S_n^r))+1)\ell(w)\\
        &\leq2(1+4\ell(w))^{\floor{\ell(w)/2}}\ell(w)(\diam(w(S_n^r))+1) \\
        &\leq 2\exp(\log(5\ell(w)) \ell(w)/2)  (\diam(w(S_n^r))+1)
    \end{align*}
    as $m\leq\ell(w)/2$ and $\ell(w)\geq 1$. This is the bound in (i) of Theorem~\ref{thm:main}. 
\end{proof}

We now turn to metric ultraproducts of symmetric groups and Corollary \ref{cor:ultra}. Assume that $l\geq 1$ and let $$w_n=x_{i(1)}^{\varepsilon_1}c_{1,n}\cdots x_{i(l)}^{\varepsilon_l}c_{l,n}\in\freegrp_r\ast S_n$$ be a sequence of reduced words of the same structure, but with different constants $c_{j,n}$ ($j=1,\ldots,l$, $n\in\nats$, $n\geq 2$). Let $\calU$ be an ultrafilter on $\nats$ such that $\diam(w_n(S_n^r))/n\to_\calU 0$ as $n\to_\calU\infty$, i.e.\ $$w=x_{i(1)}^{\varepsilon(1)}\overline{(c_{1,n})}_n\cdots x_{i(l)}^{\varepsilon(l)}\overline{(c_{l,n})}_n$$ induces a constant map on the metric ultraproduct of the $S_n$'s equipped with the normalized Hamming norm $\norm{g}_{\rm H}=\norm{g}/n$. We claim, that then there is an index $j\in\set{1,\ldots,l-1}$ such that $i(j)=i(j+1)$ and $\varepsilon(j)=-\varepsilon(j+1)$, such that $\norm{c_{j,n}}/n\to_\calU 0$ as $n\to_\calU\infty$. In particular, $w$ is trivial.

\begin{proof}[Proof of Corollary \ref{cor:ultra}]
	By the above $2\norm{w}_\infty(\diam(w_n(S_n^r))+1)\geq\norm{w}_{\mathrm c}$, so that $\norm{w}_{\rm c}/n\to_\calU 0$ as $n\to_\calU\infty$ by the assumption. Hence, by finiteness, there exists an index $j\in J_-(w)$ such that $\norm{c_{j,n}}/n\to_\calU 0$ as $n\to_\calU\infty$. 
	In particular, starting with a potential mixed identity in the metric ultraproduct, the previous argument implies that a critical constant must be trivial. This is a contradiction and the proof is complete.
\end{proof}

%\begin{proof}
%	Assume $w\not\in S_n$ induces a constant map $S_n^r\to S_n$ and that $l\leq\frac{n+1}{2}$. Then $n\geq 2l-1$ and $l\geq 2$, since $l=1$ necessarily gives a non-constant map as $n\geq 2$. Hence we have $n\geq 2l-1\geq\norm{w}_i+\norm{w}_j-1$, $n\geq 2l-1>2(l-1)\geq 2(\norm{w}_i-1)$, and $n\geq 2l-1\geq l+1\geq\norm{w}_{i(l)}+1$, so assumptions (i), (ii), and (iv) in Lemma~\ref{lem:no_lw_lrg_supp_consts} are satisfied for $d=1$. But as $J_-(w)=\emptyset$, assumption~(iii) holds as well and we can apply Lemma~\ref{lem:no_lw_lrg_supp_consts}. Hence $w$ is not constant on $S_n$ and the proof is complete.
%\end{proof}

%This last result shows that the length of a shortest strong identity for $S_n$ is greater than $\frac{n+1}{2}$ ($n\geq 2$). 

\section*{Appendix}

It is well-known that every map $\finfield_q\to\finfield_q$ is realized by a polynomial from $\finfield_q[X]$ of degree less than $q$ -- in fact this is a characterization of fields among finite rings, see  \cite{istingerkaiser1979characterization}. We prove the analogous result for non-abelian simple groups. We denote by the \emph{covering number} $\cn(G)$ of $G$ the minimal $m$, such that $C^{\ast m}=G$ for all non-trivial conjugacy classes $C$. Various bounds on $\cn(G)$ for finite simple groups $G$ can be found in the literature, see for example the seminal work of Liebeck-Shalev~\cite{liebeckshalev2001diameters} and the references therein.

\begin{theorem} \label{gensimp}
	A non-abelian finite group is simple if and only if every map $G\to G$ is a word map with constants. Every such map can be represented by a word of length at most $O(\card{G}^3 \cn(G))$.
\end{theorem}

\begin{remark}
    An abelian group $G$ has the above property (i.e.\ each map $G\to G$ comes from a word with constants) if and only if $\card{G}\leq 2$.
\end{remark}

\begin{remark}
    In the theorem, the term $\cn(G)$ can also be replaced by the \emph{covering diameter} $\cd(G)$ of $G$, which is the minimal number $m$ such that $(\set{1_G}\cup C\cup C^{-1})^{\ast m}=G$ for all non-trivial conjugacy classes $C$ of $G$. Clearly, $\cd(G)\leq\cn(G)$.
\end{remark}

\begin{proof}
	The evaluation map $\ev\colon\freegrp_1\ast G=\gensubgrp{x}\ast G\to G^G$ is a homomorpism. Also, if $\pi_{\set{g,h}}\colon G^G\to G^{\set{g,h}}$ is the projection onto $g$ and $h$, then $\pi_{\set{g,h}}\circ\ev$ is surjective: Indeed, by Goursat's lemma and since $G$ is assumed to be simple, it is enough to show that its image is not the graph of an automorphism. However, this cannot be the case, since the constants $G$ map to the diagonal subgroup, and $x$ maps to a non-diagonal subgroup. This implies that $\ev$ must be surjective too.
	
	By using iterated commutators, one can give a more explicit construction: We prove that for any subset $S\subseteq G$ of size $m\leq 2^e$ and an element $g\in G\setminus S$, there is a commutator word $w_{g,S}\in\freegrp_1\ast G$ of length $4^e$ such that the induced map $w_{g,S}\colon G\to G$ satisfies $w_{g,S}(s)=1_G$ for all $s\in S$ and $w_{g,S}(g)\neq 1_G$.
	
	Indeed, this is true when $e=0$ and $S=\set{s}$: Simply take $w_{g,\set{s}}=xs^{-1}$. For $e\geq 1$, write $S=S_1\cup S_2$ as a union of the subsets $S_1$ and $S_2$ such that $\card{S_1},\card{S_2}\leq 2^{e-1}$. Then set $w_{g,S}\coloneqq[w_{g,S_1}^a,w_{g,S_2}^b]$ where $a,b\in G$ are chosen such that $[w_{g,S_1}(g)^a,w_{g,S_2}(g)^b]\neq 1_G$. This is possible since $(w_{g,S_2}(g))^G$ generates the whole $G$ by simplicity as $w_{g,S_2}(g)\neq 1_G$, whereas $\C_G(w_{g,S_1}(g))$ is a proper subgroup. Hence $$w_{g,S_2}^G(g)\not\subseteq\C_G(w_{g,S_1}(g)).$$ We have that $\ell(w_{g,S})\leq 2(\ell(w_{g,S_1})+\ell(w_{g,S_1}))\leq 4^{e}$.
	
	Then setting $e\coloneqq\ceil{\log_2(\card{G}-1)}$, one can take $S=G\setminus\set{g}$, so that $w_{g,S}(s)=1_G$ for $s\in S$ and $w_{g,S}(g)\neq 1_G$ as desired. Now we can multiply together conjugates of the map $w_{g,S}$ to get a map $\delta_{g,h}\in\freegrp_1\ast G$ such that $\delta_{g,h}(s)=1_G$ for $s\in S=G\setminus\set{g}$ and $\delta_{g,h}(g)=h$ for a chosen $h\in G$. Then $w=\prod_{g\in G}{\delta_{g,f(g)}}=f$ as maps $w,f\colon G\to G$ where $f$ is arbitrary. So $f$ is a word map. Here $$\ell(w)\leq \card{G} 4^{\ceil{\log_2(\card{G}-1)}}\cn(G)\leq 4 \card{G}(\card{G}-1)^2\cn(G)<4\card{G}^3\cn(G).$$
	
	Conversely, if any map $f\colon G\to G$ is a word map, then for $g\neq 1_G$, the map $\delta_{g,h}\colon G\to G$ given by 
	$$
	\delta_{g,h}(x)=\begin{cases}
		1_G & \text{if $x\neq g$}\\
		h & \text{if $x=g$}
	\end{cases}
	$$
	is induced by some word $w\in\freegrp_1\ast G$. Since $w(1_G)=\delta_{g,h}(1_G)=1_G$, we have that $w(g)=\delta_{g,h}(g)=h\in\gennorsubgrp{g}$. But since $h$ was arbitrary, we have $\gennorsubgrp{g}=G$, so that $G$ is simple, since $g\neq 1_G$ was arbitrary as well.
\end{proof}

By \cite{aradstaviherzog1985powers}, Chapter~3, we have $\cn(A_n) = \floor{n/2}$ for $n\geq6$. Thus, we get the following estimate for alternating groups.

\begin{corollary}
    For $n\geq 5$, every map $A_n \to A_n$ can be represented by a word map with constants of length $O((n!)^3 n)$.
\end{corollary}

\begin{remark}
    There are at most $3^l\card{G}^{l+1}$ words with constants of length at most $l$ in $\freegrp_1\ast G=\gensubgrp{x}\ast G$. So in order that all self-maps $G\to G$ can be represented as word maps coming from words of length at most $l$, we must have 
    $$
    \card{G}^{\card{G}}\leq3^l\card{G}^{l+1}.
    $$
    This implies a linear lower bound on the minimal length of words with constants that can represent all possible maps.
\end{remark}

%For completeness, we recall the corresponding theorem for finite rings.

%\begin{theorem}
%	A finite commutative ring is simple (and so a field) iff every map $R\to R$ is induced by a polynomial $p(x)\in R[x]$.
%\end{theorem}

%\begin{proof}
%	If $R$ is a field, then $\delta_{a,1}(x)=-\prod_{r\in R\setminus\set{a}}{(x-r)}$ is a polynomial. Here 
%	$$
%	\delta_{a,b}(x)=\begin{cases} 0 & \text{if $x\neq a$}\\ b & \text{if $x=a$}\end{cases}.$$ 
%	Then $p(x)=\sum_{a\in R}{\delta_{a,f(a)}(x))}=f(x)$ is a polynomial for $f\colon R\to R$ arbitrary.
	
%	Conversely, if every map is a polynomial, look at $\delta_{a,b}(x)=\sum_{i\geq 0}{a_ix^i}$ for $a,b\neq 0$. Then         $a_0=\delta_{a,b}(0)=0$. This implies $\delta_{a,b}(a)=\sum_{i\geq 1}{a_i a^i}=b\in(a)$ and so since $b$ was arbitrary $(a)=R$. Since $a$ was arbitrary, $R$ is a field. 
%\end{proof}

\begin{remark}
    It was pointed out to us by Ben Steinberg and Anton Klyachko, that the qualitative aspect of Theorem~\ref{gensimp} is a well-known result \cite{maurerrhodes1965property}. Similar but weaker bounds have been obtained in \cite{horvathnehaniv2015length}. Using methods similar to the ones in the proof of Theorem~\ref{gensimp} one can show that every map $G^r \to G$ arises as a word map with constants of length $O(\card{G}^{r+2} r^2 \cn(G))$. Indeed, one can use the maps of length at most $O(\card{G}^2)$ to get maps that see just one coordinate of $G^r$. Taking $\ceil{\log_2(r)}$ commutators of those maps, one gets a map that sees just one element of $G^r$. This then has length $O(r^2\card{G}^2)$. Hence, we get a bound 
    $$
        O(\card{G}^{r+2} r^2\cn(G)).
    $$
    This is exactly the bound in \cite{horvathnehaniv2015length} for $A_n$, but the authors do not write it out for other non-abelian simple groups.
\end{remark}

\section*{Acknowledgments}

We thank Robert Kaak for his friendship. Moreover, we thank Christoph Schulze and Vadim Alekseev for discussions on the topic. We thank Ben Steinberg and Anton Klyachko for comments on a first draft of this paper. This research was supported by the ERC Consolidator Grant No.\ 681207.

\end{document}